\documentclass[12pt]{amsart}
\usepackage{mathtools}

\DeclarePairedDelimiterX\braket[2]{\langle}{\rangle}{#1 \delimsize\vert #2}

\usepackage{dsfont}
\usepackage{amsmath,amssymb,amsfonts,graphics,amsthm}
\usepackage{mathabx}
\usepackage{bbm}
\usepackage{bm}
\usepackage{enumerate}
\usepackage{verbatim}
\usepackage{fullpage}
\usepackage{tikz}
\usepackage{tikz-cd}
\usepackage{graphicx}
\usepackage{hyperref}
\newcommand{\N}{\mathbb{N}}

\newcommand{\R}{\mathbb{R}}

\newcommand{\C}{\mathbb{C}}

\newcommand{\HH}{\mathsf{H}}
\newcommand{\RH}{\mathsf{H}_{\mathbb{R}}}

\newcommand{\KK}{\mathsf{K}}
\newcommand{\CCH}{\overline{\mathsf{H}}}
\newcommand{\CCK}{\overline{\mathsf{K}}}

\newcommand{\op}[1]{\operatorname{#1}}

\newtheorem{thm}{Theorem}[section]
\newtheorem{cor}[thm]{Corollary}
\newtheorem{lem}[thm]{Lemma}
\newtheorem{prop}[thm]{Proposition}

\theoremstyle{definition}
\newtheorem{defn}[thm]{Definition}

\newtheorem{rem}[thm]{Remark}

\numberwithin{equation}{section}

\begin{document}
\title[CMAP for $q$-Gaussian algebras]{A simple proof of the complete metric approximation property for $q$-Gaussian algebras}
\author{Mateusz Wasilewski}
\date{}
\begin{abstract}
The aim of this note is to give a simpler proof of a result of Avsec, which states that $q$-Gaussian algebras have the complete metric approximation property.
\end{abstract}
\maketitle
\section{Introduction}
Bo\.{z}ejko and Speicher introduced $q$-Gaussian algebras in \cite{MR1105428} (see also \cite{MR1463036}). These von Neumann algebras are built from operators satisfying a deformation of canonical commutation relations. But they can also be viewed as $q$-deformations of the free group factors. It turns out that $q$-Gaussian algebras share many properties with the free group factors: they are factors (see \cite{MR2164947}), they are non-injective (see \cite{MR2091676}), they have the Haagerup approximation property (folklore; see \cite{MR3717957} for a proof in a more general setting of $q$-Araki-Woods algebras), etc. One of more important properties of von Neumann algebras studied recently is the notion of strong solidity, first introduced by Ozawa and Popa in \cite{MR2680430}; in the same paper they prove that free group factors are strongly solid. 

In an unpublished manuscript \cite{1110.4918} Avsec proved that $q$-Gaussian algebras possess the complete metric approximation property (see Theorem A therein). When combined with deformation/rigidity techniques, namely using a malleable deformation, it allowed him to also prove that $q$-Gaussian algebras are strongly solid (see Theorem B therein). We will reprove the first result, namely we will show the following.
\begin{thm}\label{Thm=Polybound}
Let $\HH_{\R}$ be a real Hilbert space and let $\Gamma_q(\HH_{\R})$ be the associated $q$-Gaussian algebra. Let $P_n\colon \Gamma_q(\HH_{\R}) \to \Gamma_q(\HH_{\R})$ be the projection onto Wick words of length $n$. Then $\|P_n\|_{\op{cb}}\leqslant C(q)n$, where $C(q)$ is a positive constant depending only on $q$.
\end{thm}
\begin{cor}\label{Cor=CMAP}
For any real Hilbert space $\HH_{\R}$ the $q$-Gaussian algebra $\Gamma_q(\HH_{\R})$ has the $w^{\ast}$-complete metric approximation property.
\end{cor}
The proof splits into two parts -- in the first one we analyse the operator space structure induced on the image of $P_n$ by inclusion into $\Gamma_q(\RH)$ and define some completely bounded maps; we will follow closely the original approach of Avsec. The second part of the proof will consist of showing that a certain linear combination of these maps is equal to the map, whose complete boundedness we want to prove (cf. \cite[Proposition 3.18]{1110.4918}). Presenting a different (and much simpler) proof of this combinatorial statement is the main aim of this note. We discuss more details in the beginning of Subsection \ref{Subsec=Combinatorial}, once all the necessary tools have been introduced; readers familiar with Avsec's paper \cite{1110.4918} may wish to jump straight to this subsection.

The structure of the paper is the following: in Section \ref{Sec=Prelim} we provide necessary information about operator spaces and $q$-Gaussian, in Section \ref{Sec=Haagerup} we recall Haagerup's argument on how to deduce Corollary \ref{Cor=CMAP} from Theorem \ref{Thm=Polybound}, finally Section \ref{Sec=Mainproof} contains the proof of Theorem \ref{Thm=Polybound}.
\subsection*{Acknowledgements}
I would like to thank Adam Skalski for careful reading of the preliminary version of this note and useful remarks. I am also thankful to \'{E}ric Ricard for pointing out that the bound on the cb norm of $P_n$ is linear in $n$, not quadratic, as I wrote in an older version of the paper. This was also spotted by the referee, whose helpful comments greatly improved the exposition.

The author was supported in part by European Research Council Consolidator Grant 614195
RIGIDITY and by long term structural funding – Methusalem grant of the Flemish Government.
\section{Notation and preliminaries}\label{Sec=Prelim}
All inner products will be linear in the second variable. We will denote the set $\{1,\dots, n\}$ by $[n]$. For simplicity we assume that the Hilbert spaces are finite dimensional -- there is a standard approximation procedure that allows us to do it.
\subsection{Operator spaces}
An \textbf{operator space} is a Banach space $X$ equipped with a sequence of norms on the matrix spaces $\op{M}_n(X)$ that satisfy natural compatibility conditions, the so-called Ruan's axioms; any such sequence comes from an embedding $X\subset \op{B}(\HH)$. Any $C^{\ast}$-algebra $A$ admits a canonical operator space structure induced by any faithful representation on a Hilbert space. For information on operator spaces we refer to the monographs \cite{MR1793753} and \cite{MR2006539}. For a linear map $T\colon X\to Y$ between operator spaces we define its \textbf{cb} norm as
\[
\|T\|_{\op{cb}}:= \sup_{n\in \N} \|\op{Id}_n\otimes T\colon \op{M}_n(X) \to \op{M}_n(Y)\|.
\] 
We say that $T$ is \textbf{completely bounded} if $\|T\|_{cb}<\infty$.

The theory of operator spaces mimicks the theory of Banach spaces. In particular, it retains one of most powerful properties of Banach spaces: the duality. For any operator space $X$ there is a naturally defined operator space structure on its Banach space dual $X^{\ast}$. Another construction is the complex conjugate operator space, defined using the identification $\op{M}_n(\overline{X})\simeq \overline{\op{M}_n(X)}$. Since for Hilbert spaces we have an isometric identification $\HH^{\ast}\simeq \HH$, dual spaces and conjugate spaces often appear together. One important fact about conjugate spaces is that for a von Neumann algebra $\mathsf{M}$ we have $\overline{\mathsf{M}} \simeq \mathsf{M}^{\op{op}}$ completely isometrically, via the map $\overline{x}\mapsto x^{\ast}$, where $\mathsf{M}^{\op{op}}$ is the same as $\mathsf{M}$ as a vector space but the multiplication is reversed i.e. $x^{\op{op}}\diamond y^{\op{op}}:= yx$. We can define an operator space structure on the spaces $L^{1}(\mathsf{M})$ for a finite von Neumann algebra $\mathsf{M}$ (equipped with a trace $\tau$) and $S^1(\mathsf{H},\KK)$ (the trace class operators) and the dualities $(L^{1}(\mathsf{M}))^{\ast}\simeq \mathsf{M}$ and $\left(S^{1}(\HH,\KK)\right)^{\ast} \simeq\op{B}(\KK,\HH)$ hold; the pairings are given by $L^{1}(\mathsf{M}) \times \mathsf{M} \ni (x,y) \mapsto \tau(xy)$ and $S^{1}(\HH,\KK)\times \op{B}(\KK,\HH) \ni (S,T) \mapsto \op{Tr}(ST)$. Some authors prefer to include complex conjugations in these dualities, i.e. consider the pairing $\overline{L^{1}(\mathsf{M})} \times \mathsf{M} \ni (\overline{x},y)\mapsto \tau(x^{\ast}y)$ (and a similar one for the trace-class operators), so that they resemble the case of the Hilbert spaces, i.e. they are ``positive definite'', but we will try to avoid them to make the notation less cumbersome; both approaches have their pros and cons.

In this paper we will encounter two special operator space structures on a Hilbert space.
\begin{defn}
Let $\HH$ be a complex Hilbert space.
\begin{enumerate}[{\normalfont (i)}]
\item The \textbf{column Hilbert space} structure $\HH_{c}$ is given by the identification $\HH \simeq \op{B}(\C, \HH)$;
\item The \textbf{row Hilbert space} structure $\HH_{r}$ is given by the identification $\HH \simeq \op{B}(\overline{\HH}, \C)$.
\end{enumerate}
\end{defn}
\begin{rem}
In particular, we have $\HH_{c}^{\ast} \simeq \overline{\HH}_{r}$ and $\HH_r^{\ast}\simeq \overline{\HH}_{c}$.
\end{rem}
These Hilbert spaces are \textbf{homogeneous}, i.e. bounded maps on the underlying Hilbert spaces are automatically completely bounded, with cb norm equal to the norm (see \cite[Theorem 3.4.1 and Proposition 3.4.2]{MR1793753}).

We are going to need the notion of a tensor product of operator spaces. The simplest one is obtained by the following procedure: we have two operator spaces $X \subset \op{B}(\HH)$ and $Y \subset \op{B}(\KK)$ and we get an operator space structure on $X\otimes Y$ via embedding $X\otimes Y \subset \op{B}(\HH\otimes \KK)$. It turns out that it does not depend on the embeddings and the completion of $X\otimes Y$ is denoted by $X\otimes_{\op{min}} Y$; it coincides with the minimal tensor product of $C^{\ast}$-algebras, in case $X$ and $Y$ are $C^{\ast}$-algebras.

There is also a special tensor product of operator spaces, called the Haagerup tensor product, which does not have a counterpart for Banach spaces; the Haagerup tensor product of two operator spaces $X$ and $Y$ will be denoted by $X\otimes_{h} Y$. One of its key properties is self-duality, i.e. $(X\otimes_{h} Y)^{\ast} \simeq X^{\ast}\otimes_{h} Y^{\ast}$ for finite dimensional operator spaces $X$ and $Y$. For the definition and more information, see \cite[Section 9]{MR1793753}. We just collect here the properties that will be useful for us in the sequel.
\begin{prop}\label{Prop=Haageruptensor}[Proposition 9.3.4 and Proposition 9.3.5 in \cite{MR1793753}]
Let $\HH$ and $\KK$ be Hilbert spaces. Then we have the following identifications:
\begin{enumerate}[{\normalfont (i)}]
\item $\HH_{c} \otimes_{h} \CCK_{r} \simeq \op{K}(\KK, \HH)$ (the compact operators);
\item $\CCK_{r} \otimes_{h} \HH_{c} \simeq \HH_{r} \otimes_{h} \CCK_{c}  \simeq S^{1}(\KK, \HH)$ (the trace class operators);
\item $ \HH_r\otimes_h \KK_r \simeq (\HH\otimes \KK)_r$ and $\HH_c\otimes_h \KK_c \simeq (\HH\otimes \KK)_c$. 
\end{enumerate}
\end{prop}
\subsection{\texorpdfstring{$q$}{q}--Gaussian algebras}
We will now define $q$-Gaussian algebras, introduced by Bo\.{z}ejko and Speicher (see \cite{MR1105428}).

Let $\RH$ be a real Hilbert space and let $\HH$ be its complexification; we will denote the complex conjugation by $I$. We would like to define the \textbf{$q$-Fock space}, which will be a certain completion of the tensor algebra $\bigoplus_{n\geqslant 0} \HH^{\otimes n}$ (both the direct sum and the tensor products are algebraic here, also $\HH^{\otimes 0} = \C\Omega$). We will encounter tensors fairly often and for convenience we will sometimes denote $v_1\otimes\dots\otimes v_n$ by $\bm{v}$ and for subset $A\subset [n]$ we will use the notation $\bm{v}_{A}:= v_{i_{1}}\otimes \dots \otimes v_{i_k}$, where $A = \{i_1<\dots<i_{k}\}$, i.e. $i_1,\dots,i_k$ are all elements of $A$ arranged in an increasing order.
\begin{defn}
Let $\HH$ be a complex Hilbert space. Fix $q\in (-1,1)$. For each $n\in \N$ we define an operator $P_q^{n}\colon \HH^{\otimes n}\to \HH^{\otimes n}$ by 
\begin{equation}
P_q^{n}(v_1\otimes\dots\otimes v_n):= \sum_{\pi \in S_n} q^{i(\pi)} v_{\pi(1)}\otimes\dots\otimes v_{\pi(n)},
\end{equation}
where $i(\pi):=|\{(i,j)\colon i<j,{ }\pi(i)>\pi(j)\}|$ is the number of inversions of the permutation $\pi$. This operator is injective and positive definite (see \cite[Proposition 1]{MR1105428}) and therefore defines an inner product on $\HH^{\otimes n}$. The $P_q^{n}$'s combine to give an inner product on $ \bigoplus_{n\geqslant 0} \HH^{\otimes n}$ and the completion of this space with respect to this inner product is called the \textbf{$q$-Fock space} and is denoted by $\mathcal{F}_q(\HH)$.
\end{defn} 
In order to define $q$-Gaussian algebras, we have to present an important class of operators on the $q$-Fock space.
\begin{defn}
Let $\xi \in \HH$. We define the \textbf{creation operator} $a_q^{\ast}(\xi)\colon \mathcal{F}_q(\HH) \to \mathcal{F}_q(\HH)$ by the formula
\begin{align*}
a_q^{\ast}(\xi)(v_1\otimes\dots\otimes v_n) &:= \xi\otimes v_1\otimes \dots\otimes v_n \\
a_q^{\ast}(\xi)(\Omega) := \xi.
\end{align*}
We define also the \textbf{annihilation operators} by $a_q(\xi):= (a_q^{\ast}(\xi))^{\ast}$. Their action on simple tensors is given by
\begin{align*}
a_q(\xi)(\Omega)=0 \\
a_q(\xi)(v_1\otimes \dots\otimes v_n)&= \sum_{i=1}^{n} q^{i-1} \langle \xi, v_{i}\rangle v_1\otimes\dots\widehat{v_i}\dots\otimes v_n,
\end{align*}
where the hat over $v_{i}$ means that this vector is omitted.
\end{defn}
These operators extend to bounded operators on $\mathcal{F}_q(\HH)$ and satisfy the $q$-commutation relations $a_q(\xi) a_q^{\ast}(\eta) - q a_q^{\ast}(\eta)a_q(\xi) = \langle \xi, \eta \rangle \op{Id}$.
\begin{defn}
Let $\RH$ be a real Hilbert space with complexification $\HH$. We define the \textbf{$q$-Gaussian algebra} $\Gamma_q(\RH)$ to be the von Neumann subalgebra of $\op{B}(\mathcal{F}_q(\HH))$ generated by the set $\{a_q^{\ast}(\xi)+a_q(\xi)\colon \xi \in \RH\}$.
\end{defn}
The vector $\Omega$ is a cyclic and separating vector for $\Gamma_q(\RH)$, moreover the corresponding vector state is a faithful trace. In particular, the $L^{2}$-space $L^{2}(\Gamma_q(\RH))$ can be identified with the Fock space $\mathcal{F}_q(\HH)$. One can show that for any simple tensor $v_1\otimes \dots \otimes v_n \in \HH^{\otimes n}$ there exists a (unique!) operator $W(v_1\otimes \dots \otimes v_n)$ such that $W(v_1\otimes \dots \otimes v_n)\Omega = v_1\otimes \dots \otimes v_n$; these operators will be called the \textbf{Wick words}. Actually, there is an explicit formula for them (see \cite[Proposition 2.7]{MR1463036}).
\begin{lem}
We have 
\begin{equation}\label{Eqn=Wickformula}
W(v_1\otimes \dots\otimes v_n) = \sum_{A \subset [n]} q^{i(A)} a_q^{\ast}(\bm{v}_{A}) a_q(I\bm{v}_{[n]\setminus A}),
\end{equation}
where for $A=\{i_1<\dots<i_k\}$ and $[n]\setminus A = \{j_{k+1}<\dots<j_n\}$ we have $i(A) := \sum_{l=1}^{k} (i_l-l)$, $a_q^{\ast}(\bm{v}_{A}):=a_q^{\ast}(v_{i_1})\dots a_q^{\ast}(v_{i_k})$, and $a_q(I\bm{v}_{[n]\setminus A}):= a_q(Iv_{j_{k+1}})\dots a_q(Iv_{j_n})$.
\end{lem}
\begin{rem}
The notations $i(A)$ and $i(\pi)$ are consistent, if we identify $A$ with a permutation $\pi_{A}$ of $[n]$ given by $\pi_{A}(l):= i_{l}$ for $l\leqslant k$ and $\pi_{A}(l):= j_{l}$ for $l>k$. Some authors identify these permutations with representatives of cosets $S_{n} \slash (S_{k}\times S_{n-k})$ with the minimal number of inversions. It is useful to think of $i(A)$ as the cost of moving the set $A$ to the left of $[n]$. First, you move $i_1$ to the first spot, so you need to make $i_1-1$ moves. Then the first spot is taken, so you move $i_2$ to the second one and the cost is $i_2-2$; proceed like that for other elements of $A$.
\end{rem}

We will need more information about the operators $P_q^{n}$; as we mentioned, they are injective and positive definite, so invertible in the finite dimensional setting. Actually, it was noted by Bo\.{z}ejko (see \cite[Theorem 6]{MR1649711}) that they are always invertible, with a bound (exponential in $n$) for the norm of the inverse provided. Therefore, whenever we have a partition $n=n_1+\dots + n_k$, we may consider $R_{n_1,\dots,n_k}^{\ast}$ -- the unique operator on $\HH^{\otimes n}$ such that $P_q^{n} = (P_q^{n_1}\otimes \dots \otimes P_q^{n_k}) R_{n_1,\dots,n_k}^{\ast}$. One can easily check that $R_{n_1,\dots,n_k}^{\ast}$ is really the adjoint of the identity map $R_{n_1,\dots,n_k}\colon \HH_{q}^{\otimes n_1}\otimes\dots\otimes \HH_q^{\otimes n_k} \to \HH_q^{\otimes n}$. In the case of $R_{n-k,k}^{\ast}$ we have an explicit formula (see \cite[Proof of Theorem 2.1]{MR1811255}):
\begin{equation}\label{Eqn=Rnk}
R_{n-k,k}^{\ast}(v_1\otimes\dots\otimes v_{n+k})= \sum_{A\subset [n], |A|=n-k} q^{i(A)} \bm{v}_{A}\otimes \bm{v}_{[n]\setminus A},
\end{equation}
with the same notation as in \eqref{Eqn=Wickformula}. Bo\.{z}ejko also proved that $\|R_{n,k}^{\ast}\|\leqslant C(q)$ as an operator on $\HH^{\otimes (n+k)}$ and, as a consequence, $P_q^{n+k} \leqslant C(q) P_q^{n} \otimes P_q^{k}$. It follows that the norm of $R_{n,k}\colon \HH_q^{\otimes n} \otimes \HH_q^{\otimes k} \to \HH_q^{\otimes (n+k)}$ is bounded by $\sqrt{C(q)}$; the same holds for the norm of $R_{n,k}^{\ast}$ as an operator from $\HH_q^{\otimes (n+k)}$ to $\HH_q^{\otimes n} \otimes \HH_q^{\otimes k}$.

For future use, we record here some equalities pertaining to operators $R_{n,k,l}^{\ast}$.
\begin{lem}\label{Lem=Splitting}
We have $R_{n,k,l}^{\ast} = ({\rm Id}_{n} \otimes R^{\ast}_{k,l}) R_{n,k+l}^{\ast} = (R_{n,k}^{\ast}\otimes {\rm Id_l}) R_{n+k,l}^{\ast}$.
\end{lem}
\begin{proof}
It follows from equalities $R_{n,k+l}({\rm Id}_n\otimes R_{k,l}) = R_{n+k,l}(R_{n,k}\otimes {\rm Id}_l) = R_{n,k,l}$.
\end{proof}
We also introduce two complex conjugations on $\HH^{\otimes n}$: $I^{n}$ given by the formula $I^{n}(v_1\otimes \dots\otimes v_n) = Iv_1\otimes\dots\otimes Iv_n$ and $J^{n}$ given by  $J^{n}(v_1\otimes\dots\otimes v_n):= Iv_n\otimes \dots \otimes Iv_1$ (this is the modular conjugation); they are both antiunitaries on $\HH_q^{\otimes n}$. They are related by $J^{n} = I^{n} \sigma_n$, where $\sigma_n(v_1\otimes \dots \otimes v_n) := v_n\otimes \dots \otimes v_1$ is a self-adjoint unitary; usually we will just write $I$ and $J$ without the superscripts. With these two conjugations we can associate two different pairings:
\begin{enumerate}
\item\label{Dualitypairing} $m_n\colon \HH^{\otimes n}\otimes \HH^{\otimes n} \to \mathbb{C}$ given by $m_n(\bm{\xi} \otimes \bm{\eta}):= \langle J^{n}\bm{\xi}, \bm{\eta}\rangle_{q}$;
\item\label{Dualitypairing2} $\widetilde{m}_n\colon \HH^{\otimes n}\otimes \HH^{\otimes n} \to \mathbb{C}$ given by $\widetilde{m}_n(\bm{\xi}\otimes \bm{\eta}):= \langle I^{n} \bm{\xi}, \bm{\eta}\rangle_q$.
\end{enumerate}
We will continue to use the same notation for pairings that involve only part of the tensor product, i.e. $m_j$ might also mean ${\rm Id}_{n-j} \otimes m_j \otimes {\rm Id}_{n-j}\colon \HH^{\otimes n}\otimes \HH^{\otimes n} \to \HH^{\otimes n-j}\otimes \HH^{\otimes n-j}$.

We can now write a nice formula for the product of two Wick words (cf. \cite[Theorem 3.3]{MR1994546}).
\begin{prop}\label{Prop=Wickproduct}
Let $\bm{\xi} \in \HH^{\otimes n}$ and $\bm{\eta} \in \HH^{\otimes k}$. Then we have
\begin{equation}\label{Eqn=Wickproduct}
W(\bm{\xi})W(\bm{\eta})\Omega = \sum_{j=0}^{\min(n,k)} m_j(R_{n-j,j}^{\ast}(\bm{\xi})\otimes R_{j,k-j}^{\ast}(\bm{\eta})).
\end{equation}
\end{prop}
\begin{proof}
We have $W(\bm{\eta})\Omega = \bm{\eta}$. By linearity of the formula we can assume that both $\bm{\xi}$ and $\bm{\eta}$ are simple tensors, i.e. $\bm{\xi}=\xi_1\otimes \dots\otimes \xi_n$ and $\bm{\eta}=\eta_1\otimes \dots\otimes \eta_k$. We can use the Wick formula \eqref{Eqn=Wickformula} to write $W(\xi_1\otimes\dots\otimes \xi_n)= \sum_{A \subset [n]} q^{i(A)} a_q^{\ast}(\bm{\xi}_{A}) a_q(I\bm{\xi}_{[n]\setminus A})$. The action of $a_q^{\ast}(\bm{\xi}_{A})$ is very simple, so we just need to understand $a_q(I\bm{\xi}_{[n]\setminus A})(\eta_1\otimes \dots \otimes \eta_k)$. Say that $|A|=n-j$. Then $a_q(I\bm{\xi}_{[n]\setminus A})(\eta_1\otimes \dots \otimes \eta_k)$ will belong to $\HH^{\otimes k-j}$; let $\bm{\mu} \in \HH^{\otimes k-j}$. We will compute the inner product $\langle \bm{\mu}, a_q(I\bm{\xi}_{[n]\setminus A})(\eta_1\otimes \dots \otimes \eta_k)\rangle_q$. Because $a_q(I\bm{\xi}_{[n]\setminus A})^{\ast}=a_q^{\ast}(J\bm{\xi}_{[n]\setminus A})$ (adjoint reverses the order), we get that this is equal to $\langle J\bm{\xi}_{[n]\setminus A}\otimes \bm{\mu}, \bm{\eta}\rangle_q$. It would be easier to compute this inner product in $\HH_q^{\otimes j}\otimes \HH_q^{\otimes k-j}$ instead of $\HH_q^{\otimes k}$. Because of the formula $P_q^k = (P_q^{j}\otimes P_q^{k-j})R^{\ast}_{j,k-j}$, it is possible to switch between the two, at the expense of applying $R^{\ast}_{j,k-j}$ to $\bm{\eta}$. This gives us
\[
\langle J\bm{\xi}_{[n]\setminus A}\otimes \bm{\mu}, \bm{\eta}\rangle_q = \widetilde{m}_{k-j}m_{j} (I\bm{\mu}\otimes\bm{\xi}_{[n]\setminus A} )\otimes R^{\ast}_{j,k-j}(\bm{\eta}).
\]
It follows that $a_q(I\bm{\xi}_{[n]\setminus A})\bm{\eta}= m_j(\bm{\xi}_{[n]\setminus A}\otimes R^{\ast}_{j,k-j}(\bm{\eta}))$. We can finish the proof by invoking the formula $R^{\ast}_{n-j,j} = \sum_{A \subset [n], |A|=n-j} q^{i(A)} \bm{\xi}_{A}\otimes \bm{\xi}_{[n]\setminus A}$, combined with the formula for $a_q^{\ast}(\bm{\xi}_{A})$.
\end{proof}
We will need one more ingredient, crucial for constructing approximating maps on the $q$-Gaussian algebras.
\begin{prop}[{\cite[Theorem 2.11]{MR1463036}}]
Let $T\colon \RH \to \RH$ be a contraction on a real Hilbert space. Then there exists a unique map on $\Gamma_q(\RH)$, called the \textbf{second quantisation} of $T$ and denoted by $\Gamma_q(T)$, which satisfies $\Gamma_q(T)(W(\xi_1\otimes \dots\otimes \xi_n)) = W(T\xi_1\otimes\dots\otimes T\xi_n)$. Moreover, this map is unital, completely positive and trace-preserving.
\end{prop}
\section{CMAP and Haagerup's argument}\label{Sec=Haagerup}
In this short section we discuss why Theorem \ref{Thm=Polybound} implies Corollary \ref{Cor=CMAP}, basing on a classical argument of Haagerup. We first need to define the $w^{\ast}$-complete metric approximation property.
\begin{defn}
Let $\mathsf{M}$ be a von Neumann algebra. We say that it has the \textbf{$w^{\ast}$-complete metric approximation property} if there exists a net of maps $\{T_i\colon \mathsf{M}\to\mathsf{M}\}_{i\in I}$ that are finite rank, completely contractive, and $\lim_{i\in I} T_{i}(x) = x$ in the $w^{\ast}$-topology for any $x\in \mathsf{M}$.
\end{defn}
We denote by $P_n\colon \Gamma_q(\RH) \to \Gamma_q(\RH)$ the projection onto Wick words of length $n$, i.e. the operator $P_n(W(\xi_1\otimes \dots \otimes \xi_m) = \delta_{nm} W(\xi_1\otimes \dots \otimes \xi_m)$; as we will discuss in the next section, it extends to a continuous map on $\Gamma_q(\RH)$.
\begin{proof}[Proof of Corollary \ref{Cor=CMAP}]
We assume that Theorem \ref{Thm=Polybound} holds, i.e. $\|P_n\|_{\op{cb}} \leqslant C(q) n$. We need to define the net of approximating maps. We consider $T_{n,t}:= \Gamma_q(e^{-t}) Q_n$, where $Q_n:=\sum_{k\leqslant n} P_k$ is the projection onto words of length at most $n$. These maps are finite rank (recall that we assume $\dim \RH <\infty$) and we would like to check that they are (almost) completely contractive, if we let $n$ depend on $t$. We have
\[
\|T_{n,t}\|_{\op{cb}} \leqslant \|\Gamma_q(e^{-t})\|_{\op{cb}} + \|\Gamma_q(e^{-t})({\rm Id}-Q_n)\|_{\op{cb}}
\]
by the triangle inequality. Clearly $\|\Gamma_q(e^{-t})\|_{\op{cb}}\leqslant 1$ and $\Gamma_q(e^{-t})({\rm Id}-Q_n) = \sum_{k=n+1}^{\infty} e^{-kt} P_k$, so $\|\Gamma_q(e^{-t})({\rm Id}-Q_n)\|_{\op{cb}} \leqslant C(q) \sum_{k=n+1} e^{-kt} k$. This is a tail of a convergent series, so we can make it arbitrarily small if we let $n$ be large enough. It will give a bound $\|T_{n,t}\|_{\op{cb}} \leqslant 1+\varepsilon$, so the maps $S_{n,t}:=\frac{T_{n,t}}{\|T_{n,t}\|_{\op{cb}}}$ are completely contractive and for appropriate choice of $n$ and $t$ they do not differ much from $T_{n,t}$. In order to check convergence $\lim S_{n,t} x = x$ it suffices to check it for $T_{n,t}$. Note that the operators $T_{n,t}$ induce contractions on the level of the $L^{2}$-space, i.e. on the Fock space $\mathcal{F}_q(\HH)$. Since $T_{n,t}x$ sits inside the unit ball, if it converges in the $L^{2}$-norm, it also converges strongly, hence ultraweakly; it suffices to prove the convergence in the $L^{2}$-norm. But the operators $T_{n,t}$ are uniformly bounded, so it is enough to check this convergence on a dense subset, and we can choose tensors of finite rank as such a subspace; it is very simple to check the convergence there.
\end{proof}
\section{Proof of the main result}\label{Sec=Mainproof}
We would like to show now that the projection $P_n\colon \Gamma_q(\RH) \to \Gamma_q(\RH)$ is completely bounded, and the bound on the cb norm is polynomial in $n$. In order to proceed, we need a notation for the image of $P_n$ -- we will denote the space of Wick words of length $n$ by $X_n$. We will first deal with the operator space theoretic considerations and then go on straight to the proof of the combinatorial formula presented in Proposition \ref{Prop=Combinatorial}. The consistent use of properties of the operators $R_{n,k}^{\ast}$ (cf. \eqref{Eqn=Rnk}) instead of combinatorics of pair partitions will be key to obtaining a simple proof.

\subsection{Operator space of Wick words of length $n$ }
We start with a non-commutative Khintchine inequality obtained by Nou (see \cite[Theorem 1]{MR2091676}), which provides the operator space structure of $X_n$. 
\begin{thm}
Let $\bm{\xi} \in \op{B}(\KK)\otimes \HH^{\otimes n}$. Then we have 
\begin{align}\label{Eqn=Khintchine}
\max_{0\leqslant k \leqslant n}\|({\rm Id}\otimes R^{\ast}_{n-k,k})(\bm{\xi})\| &\leqslant \|({\rm Id}\otimes W)(\bm{\xi})\| \\
&\leqslant C(q)(n+1)\max_{0\leqslant k \leqslant n}\|({\rm Id}\otimes R^{\ast}_{n-k,k})(\bm{\xi})\|, \notag
\end{align}
where the norm $\|({\rm Id}\otimes R^{\ast}_{n-k,k})(\bm{\xi})\|$ is computed in $\op{B}(\KK)\otimes_{\op{min}} (\HH_q^{\otimes n-k})_{c}\otimes_{h} (\HH_q^{\otimes k})_r$.
\end{thm}
It means that the map $\iota_{n}\colon X_n \to Y_n:= \bigoplus_{k=0}^{n} (\HH_q^{\otimes n-k})_{c}\otimes_{h} (\HH_q^{\otimes k})_r$ given by $W(\bm{\xi})\mapsto (R^{\ast}_{n-k,k}(\bm{\xi}))_{0\leqslant k\leqslant n}$ is a completely isomorphic embedding, with distortion at most $C(q)(n+1)$. Therefore it suffices to prove that the map $\iota_n\circ P_n$ is completely bounded. It will actually be easier to prove that the predual of this map is completely bounded; it is not immediately clear that this map should admit a predual, but we will write it down explicitly. So we will work with the map
\[
\widetilde{\Phi}_n\colon \ell^1-\bigoplus_{k=0}^{n}(\CCH_q^{\otimes n-k})_{r}\otimes_h (\CCH_{q}^{\otimes k})_{c} \to L^{1}(\Gamma_q(\RH))   
\]
given by $\widetilde{\Phi}_n(\overline{\bm{\xi}}_0,\dots,\overline{\bm{\xi}}_n):= \sum_{k=0}^{n} \widetilde{\Phi}_{n-k,k}(\overline{\bm{\xi}}_k)$ with $\widetilde{\Phi}_{n-k,k}(\overline{\bm{\xi}}):= (W(R_{n-k,k}(\bm{\xi})))^{\ast}$. Recall that 
\[
R_{n-k,k}(\xi_1\otimes \dots \otimes \xi_{n-k}\otimes_h \xi_{n-k+1}\otimes \dots \otimes \xi_n) = \xi_1\otimes\dots\otimes \xi_n.
\]
We can use the identification $L^{1}((\Gamma_q(\RH)^{\op{op}})) \simeq \overline{L^{1}(\Gamma_q(\RH))}$ given by $x \mapsto \overline{x^{\ast}}$ and take the complex conjugate to obtain a map
\[
\Phi_n \colon \ell^1-\bigoplus_{k=0}^{n}(\HH_q^{\otimes n-k})_{r}\otimes_h (\HH_{q}^{\otimes k})_{c} \to L^{1}((\Gamma_q(\RH))^{\op{op}})  
\]
given by $\Phi_n(\bm{\xi}_0,\dots,\bm{\xi}_n)= \sum_{k=0}^{n} \Phi_{n-k,k}(\bm{\xi}_k)$, where $\Phi_{n-k,k}(\bm{\xi})= W(R_{n-k,k}(\bm{\xi}))$.
\begin{rem}\label{Rem=reduction}
In order to prove Theorem \ref{Thm=Polybound}, it is enough to prove that $\|\Phi_{n-k,k}\|_{cb} \leqslant C(q)$ for any $n$ and $k$, because this will imply that $\|\Phi_{n}\|_{\op{cb}}\leqslant C(q)$; the reason is that $\Phi_n$ is a map defined on an $\ell^1$-direct sum. Moreover, it implies that $\|\iota_n \circ P_n\|_{\op{cb}} \leqslant C(q)$, hence $\|P_n\|_{\op{cb}} \leqslant C'(q) (n+1)\leqslant 2C'(q) n$, as $\iota_n$ is an isomorphic embedding with distortion controlled by a multiple of $(n+1)$.
\end{rem}
Since it will simplify the notation, we will work with the maps $\Phi_{n,k}$. We will write them down in terms of some other maps, whose complete boundedness is easy to check. 
\begin{lem}[{\cite[Lemma 3.4]{1110.4918}}]
Let $v_{n,k}\colon (\HH_q^{\otimes n})_{r} \otimes_{h} (\HH_q^{\otimes k})_{c} \to L^{1}((\Gamma_q(\RH))^{\op{op}})$ be given by
\[
v_{n,k}(\xi_1\otimes \dots \otimes \xi_n \otimes_h \xi_{n+1}\otimes \dots \otimes \xi_{n+k}) := W(\xi_1\otimes\dots \otimes \xi_n) W(\xi_{n+1}\otimes \dots \otimes \dots \xi_{n+k}).
\]
Then $\|v_{n,k}\|_{\op{cb}} \leqslant 1$.
\end{lem}
\begin{proof}
We can view $v_{n,k}$ as a restriction of the multiplication map on $\mathcal{F}_q(\HH)_{r}\otimes_{h} \mathcal{F}_q(\HH)_{c}$. It will be convenient to view $\mathcal{F}_q(\HH)$ as the $L^{2}$-space of $\Gamma_q(\RH)$. By Proposition \ref{Prop=Haageruptensor} we have the identification $ \left(\overline{L^{2}(\Gamma_q(\RH))}\right)_{r} \otimes_{h} \left(L^{2}(\Gamma_q(\RH))\right)_{c} \simeq S^1(\mathcal{F}_q(\HH))$ given by (the extension of) $\overline{x} \otimes_{h} y \mapsto |x\rangle\langle y|$. It is not hard to check that the map from $\left(\overline{L^{2}(\Gamma_q(\RH))}\right)_{r} \otimes_{h} \left(L^{2}(\Gamma_q(\RH))\right)_{c}$ to $L^{1}(\Gamma_q(\RH))$ given by $\overline{x} \otimes_{h} y \mapsto x^{\ast}y$ is precisely the predual of the inclusion $(\Gamma_q(\RH))^{\op{op}} \hookrightarrow \op{B}(\mathcal{F}_q(\HH))$, given by the right action (traciality of the vacuum state is important here). We get a map from $L^{2}(\Gamma_q(\RH)) \otimes_{h} L^{2}(\Gamma_q(\RH))$ to $L^{1}(\Gamma_q(\RH))$ by using the identification $\left(L^{2}(\Gamma_q(\RH))\right)_{r} \simeq \left(\overline{L^{2}(\Gamma_q(\RH))}\right)_{r}$ induced by $x \mapsto \overline{x^{\ast}}$.
\end{proof}
As mentioned in the introduction, the important combinatorial input will be a formula presenting a Wick word in terms of products of shorter Wick words. This aim will be achieved using the following maps.
\begin{lem}\label{Lem=ShortWick}
Let $w^{j}_{n,k}\colon (\HH_q^{\otimes n})_{r} \otimes_{h} (\HH_q^{\otimes k})_{c} \to L^{1}((\Gamma_q(\RH))^{\op{op}})$ (for $j\leq \min(n,k)$) be given by 
\[
w^{j}_{n,k} = v_{n-j,k-j} \circ (\widetilde{m}_j (R^{\ast}_{n-j,j} \otimes R^{\ast}_{j,k-j})).
\]
Then $\|w^{j}_{n,k}\|_{\op{cb}} \leqslant D(q)$.
\end{lem}
\begin{proof}
By Proposition \ref{Prop=Haageruptensor} we know that $\left(\HH_q^{\otimes n-j}\right)_{r} \otimes_{h} \left(\HH_q^{\otimes j}\right)_{r} \simeq \left(\HH_q^{\otimes n-j} \otimes \HH_q^{\otimes j}\right)_{r}$ (analogously for the column spaces). It follows that the maps $R^{\ast}_{n-j,j}\colon \left(\HH_q^{\otimes n}\right)_{r} \to \left(\HH_q^{\otimes n-j}\right)_{r} \otimes_{h} \left(\HH_q^{\otimes j}\right)_{r}$ and $R^{\ast}_{j,k-j}\colon \left(\HH_q^{\otimes k}\right)_{c} \to \left(\HH_q^{\otimes j}\right)_{c} \otimes_{h} \left(\HH_q^{\otimes k- j}\right)_{c} $ are completely bounded with cb norms bounded by $C(q)$. The map $\widetilde{m}_j\colon \left(\HH_q^{\otimes j}\right)_{r} \otimes_{h} \left(\HH_q^{\otimes j}\right)_{c} \to \mathbb{C}$ is the duality pairing \eqref{Dualitypairing2}. The conjugation $I^{j}\colon \HH_q^{\otimes j} \to \CCH_{q}^{\otimes j}$ is a linear isometry, so it also a complete isometry between the corresponding row Hilbert spaces. It follows that the cb norm of $\widetilde{m}_j$ is equal to the cb norm of the inner product, viewed as a map from $\left(\CCH_q^{\otimes j}\right)_{r} \otimes_{h} \left(\HH_q^{\otimes j}\right)_{c}$. Since $\left(\CCH_q^{\otimes j}\right)_{r} \otimes_{h} \left(\HH_q^{\otimes j}\right)_{c}\simeq S^1(\HH_q^{\otimes j})$ and the pairing is then equal to the trace, it is completely contractive. The map $v_{n-j, k-j}$ is completely contractive by the previous lemma, so we get $\|w^{j}_{n,k}\|_{\op{cb}} \leqslant D(q):= C(q)^2$.
\end{proof}
In the next subsection we will show that $\Phi_{n,k}= \sum_{j=0}^{\min(n,k)} \alpha_j w^{j}_{n,k}$ for an appropriate choice of scalars $\alpha_j$.
\subsection{The combinatorial formula}\label{Subsec=Combinatorial}

Before we proceed, it is in order to discuss the main differences between our approach and Avsec's (see \cite[Remark 3.7--Claim 3.19]{1110.4918} for his proof). We employ properties of the operators $P_q^{n}$, defining the $q$-deformed inner products, and especially the related operators $R_{n,k}^{\ast}$ and $R_{n,k,l}^{\ast}$ to harness the combinatorics; a simple case of this procedure can be seen in the formula \ref{Eqn=Wickproduct}, which usually features a sum over a certain type of pair partitions. This replaces the intricate analysis of pair partitions performed by Avsec, which included  introduction of a new crossing number, based on a way in which a given pair partition can be build from smaller pair partitions. We, in turn, use simple algebraic properties of the aforementioned operators to arrive at the formula \eqref{Eqn=Simplified}, in which dependence on the variable $j$ has been drastically diminished. This allows us to conclude by using a very elementary Lemma \ref{Lem=Permutations}. We also completely avoid the use of ultraproduct embeddings and ``colour'' operators (see \cite[Definition 3.11]{1110.4918}). To sum up, our argument uses less sophisticated tools, and the algebraic manipulations involved are fairly elementary and not too involved; the result is a significantly shorter proof.

Let us get back to the study of the maps $w^{j}_{n,k}$. Previously we had to be careful with certain operator space theoretic identifications, because we had to ensure complete boundedness of certain maps. Now our only concern is an algebraic equality, so we will drop most of the decorations. So now $w^{j}_{n,k}$ will be treated as a map from $\HH^{\otimes n+k}$ to $\mathcal{F}_q(\HH)$, where $W(\bm{\xi}) \in L^{1}((\Gamma_q(\RH))^{\op{op}})$ is identified with the corresponding tensor $\bm{\xi}$. Since $w^{j}_{n,k}$ gives as an output a product of two Wick words, we will use the formula \eqref{Eqn=Wickproduct} to make the result of applying $w^{j}_{n,k}$ to a tensor more explicit:
\begin{align*}
w^{j}_{n,k} (\bm{\xi}\otimes \bm{\eta}) &= \sum_{0\leqslant s\leqslant \min(n,k)-j} m_s(R^{\ast}_{n-j-s,s}\otimes \widetilde{m}_j\otimes R^{\ast}_{s,k-j-s})(R^{\ast}_{n-j,j}(\bm{\xi})\otimes R_{j,k-j}^{\ast}(\bm{\eta})).
\end{align*}
Now we can move the $\widetilde{m}_j$ to the left, so that each summand in our formula is of the form
\[
(R^{\ast}_{n-j-s,s}\otimes {\rm Id}_j \otimes{\rm Id}_j\otimes R^{\ast}_{s,k-j-s})(R^{\ast}_{n-j,j}(\bm{\xi})\otimes R_{j,k-j}^{\ast}(\bm{\xi}))
\]
followed by the pairing $\widetilde{m}_j$ and then $m_s$. By Lemma \ref{Lem=Splitting} we can therefore write the end result as
\[
w^{j}_{n,k} (\bm{\xi}\otimes \bm{\eta}) = \sum_{0\leqslant s\leqslant \min(n,k)-j} m_s \widetilde{m}_j (R^{\ast}_{n-j-s,s,j}(\bm{\xi})\otimes R^{\ast}_{j,s,k-j-s}(\bm{\eta})).
\]
This is already nice, but we can make it even nicer by introducing a new variable $p=j+s$. Our next aim is to transform  the two pairings $m_s$ and  $\widetilde{m}_j$ into a single pairing $m_{s+j}$. In order to do that, we recall the formula $\widetilde{m}_j = m_j ({\rm Id} \otimes \sigma_j)$, so we can write
\[
w^{j}_{n,k} (\bm{\xi}\otimes \bm{\eta}) = \sum_{0\leqslant s\leqslant \min(n,k)-j} m_s m_j (R^{\ast}_{n-j-s,s,j}(\bm{\xi})\otimes (\sigma_j \otimes {\rm Id}_{k-j}) R^{\ast}_{j,s,k-j-s}(\bm{\eta})).
\]
Now we want to compare the pairings $m_s m_j$ and $m_{s+j}$. The former is actually computing the inner product in $H_q^{\otimes s}\otimes H_q^{\otimes j}$ and the latter in $H_q^{\otimes s+j}$. We know that the two are related by the operator $R_{s,j}^{\ast}$. We have to be slightly careful, because the pairings involve not only the inner product, but also the complex conjugation $J$, which reverses the order. Taking that into account, we obtain a formula
\[
m_{s+j} = m_s m_j(R^{\ast}_{s,j}\otimes {\rm Id}_{s+j}).
\]
We can now use Lemma \ref{Lem=Splitting} to write $R^{\ast}_{n-j-s,s,j} = ({\rm Id}_{n-p} \otimes R^{\ast}_{s,j}) R^{\ast}_{n-p,p}$. Similarly, we get $R^{\ast}_{j,s,k-j-s} = (R^{\ast}_{j,s} \otimes {\rm Id}_{k-p}) R^{\ast}_{p,k-p}$. Therefore
\begin{equation}\label{Eqn=Simplified}
w^{j}_{n,k} (\bm{\xi}\otimes \bm{\eta})= \sum_{j\leqslant p\leqslant \min(n,k)} m_{p} (R^{\ast}_{n-p,p}(\bm{\xi})\otimes ((\sigma_j\otimes{\rm Id}_{p-j})R^{\ast}_{j,p-j}\otimes {\rm Id}_{k-p})R^{\ast}_{p,k-p}(\bm{\eta})).
\end{equation}
Now the variable $j$ only appears in a few places. Our goal was to obtain a formula ${\rm Id}_{n+k} =\sum_{j} \alpha_j w^{j}_{n,k}$; we can write the right hand side as
\begin{equation}
\sum_{j}\alpha_j w^{j}_{n,k} = \sum_{0\leqslant p\leqslant \min(n,k)} m_{p} ({\rm Id}_n \otimes \left(\sum_{j\leqslant p}\alpha_j(\sigma_j\otimes {\rm Id}_{p-j}) R^{\ast}_{j,p-j}\right)\otimes {\rm Id}_{k-p})(R^{\ast}_{n-p,p}\otimes R^{\ast}_{p,k-p}).
\end{equation}
Note that for different $p$'s we get tensors of different length, so it is necessary and sufficient to check that $\sum_{j\leqslant p} \alpha_j(\sigma_j\otimes {\rm Id}_{p-j}) R^{\ast}_{j,p-j}$ is equal to ${\rm Id}_{n+k}$ for $p=0$ (which is easy, just take $\alpha_0=1$) and equal to $0$ for $p\geqslant 1$. This is what we plan to do.

Recall that the operator $R_{j,p-j}^{\ast}$ is a weighted sum of certain permutations, where the weights are of the form $q^{i(\pi)}$. Let us be more specific. If $A=\{i_1<\dots<i_j\}$ then the permutation is given by $\pi(l)=i_l$ for $l\leqslant j$ and it is increasing for $l>j$. The number of inversions is then $i(\pi) = \sum_{l=1}^{j} (i_l-l)$. If we then apply $\sigma_j\otimes {\rm Id}_{p-j}$, the permutation is $\widetilde{\pi}(l) = i_{j-l+1}$ for $l\leqslant j$ and at stays the same for $l>j$. So we get permutations, whose restrictions to $\{1,\dots,j\}$ are decreasing and restrictions to $\{j+1,\dots,p\}$ are increasing.

It turns out that a permutation of this form can arise only in two ways.
\begin{lem}\label{Lem=Permutations}
Let $S^{j}_{n}$ be the set of permutations of $[n]$ such that the restriction to $\{1,\dots,j\}$ is decreasing and the restriction to $\{j+1,\dots,n\}$ is increasing. Let $\pi \in S^{j}_{n}$. Suppose that $\sigma \in S^{k}_{n}$ (for $k\neq j$) satisfies $\sigma = \pi$. Then $k=j-1$ or $k=j+1$ and only one of these situations occurs.
\end{lem}
\begin{proof}
Suppose that $\sigma \in S^{k}_{n}$ for $k\geqslant j+2$. Then $\sigma$ restricted to $\{j+1,\dots, n\}$ is not increasing. We argue similarly for $k \leqslant j-2$.

Now let $\pi$ be given by a subset $A=\{i_1<\dots<i_j\}$. If $i_1=1$ then there cannot be a $\sigma \in S_{n}^{j+1}$ such that $\sigma=\pi$. Indeed, we have $\sigma(j)=1$, so $\sigma(j+1)>\sigma(j)$, even though $\sigma$ was supposed to decrease on $\{1,\dots,j+1\}$. On the other hand, there is a permutation $\sigma \in S^{j-1}_{n}$ such that $\sigma=\pi$; simply let $\sigma(j)=1$ and it will be increasing on $\{j,\dots,n\}$. In case $i_1>1$ there is a unique permutation $\sigma \in S^{j+1}_{n}$ such that $\sigma=\pi$; the proof is similar.
\end{proof}
What remains to be done is the appropriate choice of coefficients $\alpha_j$. We need cancellations, so the signs must alternate and we will work with powers of $q$, so it will be easier to work with a coefficient $\beta_j$ such that $\alpha_j = (-1)^{j} q^{\beta_j}$. Suppose that we have a set $A=\{i_1<\dots<i_j\}$ with $i(A) = \sum_{l=1}^{j} (i_l-l)$. In case $i_1=1$ the corresponding permutation arises also from the subset $A':=\{i_2<\dots<i_j\}$ with $i(A') = \sum_{l=2}^{j} (i_l - (l-1)) = \sum_{l=1}^{j} (i_l-l) + (j-1)=i(A)+j-1$. If we want the cancellation, it follows that $\beta_{j-1} + (j-1) = \beta_j$. Since we know that $\beta_0=0$, it follows that $\beta_j = {j\choose 2}$. We just have to check that these coefficients also work in the case $1 \notin A$. Then the set $\widetilde{A}:= \{1<i_1<\dots< i_j\}$ yields the same permutation and $i(\widetilde{A}) = \sum_{l=1}^{j} (i_l - (l+1)) = i(A) - j$. So the compatibility condition in this case is $\beta_{j+1} - j = \beta_{j}$, i.e. $\beta_{j+1} = \beta_j+j$, which is satisfied by our choice. Thus we have proved the following proposition.
\begin{prop}\label{Prop=Combinatorial}
We have $\Phi_{n,k} = \sum_{j=0}^{\min(n,k)} (-1)^{j} q^{{j \choose 2}} w^{j}_{n,k}$.
\end{prop}
\begin{proof}[Proof of Theorem \ref{Thm=Polybound}]
We want to find a bound for the cb norm of the projection $P_n$ onto Wick words of length $n$. The Khintchine inequality \eqref{Eqn=Khintchine} shows that this subspace is completely isomorphic (with distortion linear in $n$) to a certain well behaved operator space; a direct sum of Haagerup tensor products of row and column Hilbert spaces. A modification of the predual of this map is $\Phi_n$, which is defined using $\Phi_{n,k}$ (see Remark \ref{Rem=reduction}); it suffices to prove that $\|\Phi_{n,k}\|_{\op{cb}}$ is bounded by a constant depending only on $q$. By the previous proposition we have 
\[
\|\Phi_{n,k}\|_{\op{cb}} \leqslant \sum_{j=0}^{\min(n,k)} |q|^{{j\choose 2}} \|w^{j}_{n,k}\|_{\op{cb}} \leqslant \underbrace{\left(\sum_{j=0}^{\infty} |q|^{{j\choose 2}}\right)}_{=C'(q)} D(q),
\]
where the bound for $\|w^{j}_{n,k}\|_{\op{cb}}$ comes from Lemma \ref{Lem=ShortWick}.
\end{proof}
\bibliographystyle{alpha}
\bibliography{Research}

\end{document}